\documentclass{article}

\usepackage[T1]{fontenc}
\usepackage[utf8]{inputenc}

\usepackage[margin=1.4in]{geometry}
\usepackage{graphicx}

\usepackage{amssymb,amsthm}
\usepackage{mathtools}
\numberwithin{equation}{section}

\usepackage[
natbib,
maxcitenames=3, 
maxbibnames=10,  
sorting=ydnt,
url=false,
doi=false,
sortcites,
defernumbers,
backref,
backend=biber
]{biblatex}
\usepackage{hyperref}

\usepackage{todonotes}
\usepackage{url}

\usepackage[nameinlink, capitalise, noabbrev]{cleveref}

\usepackage{xfrac}
\usepackage{nicefrac}

\newtheorem{theorem}{Theorem}

\newtheorem{proposition}{Proposition}[section]
\newtheorem{lemma}{Lemma}[section]

\theoremstyle{remark}
\newtheorem{remark}{Remark}[section]

\theoremstyle{definition}
\newtheorem{example}{Example}[section]
\newtheorem{definition}{Definition}[section]

\numberwithin{equation}{section}

\newcommand{\R}{\mathbb{R}}

\newcommand{\closure}[1]{\overline{#1}}
\newcommand{\norm}[1]{\left\Vert #1 \right\Vert}
\newcommand{\seminorm}[1]{\left[ #1 \right]}
\newcommand{\abs}[1]{\left\vert #1 \right\vert}
\DeclareMathOperator{\divtmp}{div}
\renewcommand{\div}{\divtmp}

\DeclareMathOperator{\essinf}{ess\,inf}
\newcommand{\st}{\,:\,}

\newcommand{\dx}{\,\mathrm{d}x}
\renewcommand{\d}{\,\mathrm{d}}

\newcommand{\eps}{\varepsilon}
\DeclareMathOperator{\dist}{dist}
\DeclareMathOperator{\Lip}{Lip}

\newcommand{\C}{\mathrm{C}}
\renewcommand{\L}{\mathrm{L}}
\newcommand{\W}{\mathrm{W}}

\newcommand{\grad}{\nabla}
\newcommand{\hess}{\mathrm{D}^2}

\DeclareMathOperator{\diam}{diam}
\newcommand{\Set}[1]{\left\lbrace#1\right\rbrace}

\newcommand{\rev}{\color{magenta}}
\renewcommand{\rev}{}

\newcommand{\nc}{\normalcolor}

\makeatletter
\let\blx@rerun@biber\relax
\makeatother

\begin{document}

\title{The convergence rate of $p$-harmonic\\to infinity-harmonic functions}
\author{Leon Bungert\thanks{Institute of Mathematics, University of Würzburg, Emil-Fischer-Str. 40, 97074 Würzburg, Germany. Email: \href{mailto:leon.bungert@uni-wuerzburg.de}{leon.bungert@uni-wuerzburg.de}
} 
}
\maketitle

\begin{abstract}
    The purpose of this paper is to prove a uniform convergence rate of the solutions of the $p$-Laplace equation $\Delta_p u = 0$ with Dirichlet boundary conditions to the solution of the infinity-Laplace equation $\Delta_\infty u = 0$ as $p\to\infty$.
    The rate scales like $p^{-\nicefrac{1}{4}}$ for general solutions of the Dirichlet problem and like $p^{-\nicefrac{1}{2}}$ for solutions with positive gradient.
    An explicit example shows that it cannot be better than $p^{-1}$.
    The proof of this result solely relies on the comparison principle with the fundamental solutions of the $p$-Laplace and the infinity-Laplace equation, respectively.
    Our argument does not use viscosity solutions, is purely metric, and is therefore generalizable to more general settings where a comparison principle with Hölder cones and Hölder regularity is available.
\end{abstract}

\section{Introduction}

The infinity-Laplace equation first appeared in \cite{aronsson1967extension} as optimality condition for absolutely minimizing Lipschitz extensions.
There it was formally derived as the limit equation of $p$-Laplacian problems of the form
\begin{align}\label{eq:p-Laplace_equation}
    \begin{dcases}
        \Delta_p u = 0\quad&\text{in }\Omega,\\
        u = g_p\quad&\text{on }\partial\Omega,
    \end{dcases}
\end{align}
where the $p$-Laplace operator of a smooth function $u$ is defined as $\Delta_p u := \div(\abs{\grad u}^{p-2}\grad u)$, $\Omega\subset\R^d$ is a bounded domain, and $g_p:\partial\Omega\to\R$ is some boundary datum.
To derive the infinity-Laplace equation, one expands the $p$-Laplacian as 
\begin{align}
    \Delta_p u = \abs{\grad u}^{p-4}\left(\abs{\grad u}^2\Delta u + (p-2)\langle\grad u,\hess u\,\grad u\rangle\right).
\end{align}
Using the homogeneity of the equation \labelcref{eq:p-Laplace_equation} and carelessly dividing by $\abs{\grad u}$, the right limit equation as $p\to\infty$ turns out to be
\begin{align}\label{eq:inf-Laplace_equation}
    \begin{dcases}
        \Delta_\infty u = 0\quad&\text{in }\Omega,\\
        u = g_\infty\quad&\text{on }\partial\Omega,
    \end{dcases}
\end{align}
where the infinity-Laplace operator is defined as $\Delta_\infty u := \langle\grad u,\hess u\,\grad u\rangle$, and $g_\infty:\partial\Omega\to\R$ is a Lipschitz continuous boundary datum.
One can make this limit rigorous using the framework or viscosity solutions, see, e.g., \cite{aronsson2004tour,lindqvist2016notes}.
Alternatively, one can consider a variational limit of the $p$-Laplacian problems, given by the absolutely minimizing Lipschitz extension problem:
\begin{align}\label{eq:AMLE}
    \begin{dcases}
        \Lip(u;U) = \Lip(u;\partial U)\quad&\forall \text{ subdomains }U \subset \Omega,\\
        u = g_\infty\quad&\text{on }\partial\Omega.
    \end{dcases}
\end{align}
Here $\Lip(u;X) := \sup_{x,y\in X}\frac{\abs{u(x)-u(y)}}{\abs{x-y}}$ is the Lipschitz constant of $u$ on $X\subset\closure\Omega$.
While in Euclidean space \labelcref{eq:inf-Laplace_equation} is equivalent to \labelcref{eq:AMLE} (see \cite{aronsson2004tour}), the latter formulation can be extended to more general settings like graphs \cite{bungert2022uniform,legruyer2007absolutely,roith2022continuum} or length spaces \cite{juutinen2006equivalence,juutinen2002absolutely}, where a PDE like \labelcref{eq:inf-Laplace_equation} might not be available.
Yet another equivalent formulation of \labelcref{eq:inf-Laplace_equation,eq:AMLE} is the ``comparison with cones'' principle \cite{aronsson2004tour}.
It states that a function solves these two problems if and only if it admits a comparison principle with cone functions of the form $x\mapsto a\abs{x-x_0}+b$.
While the necessity of this property is not really surprising (given that all solutions of the $p$-Laplace equation \labelcref{eq:p-Laplace_equation} for $p>d$ also admit a comparison principle with functions of the form $x\mapsto a\abs{x-x_0}^\frac{p-d}{p-1}+b$) the sufficiency is an astonishing feature of the infinity-Laplace equation~\labelcref{eq:inf-Laplace_equation}.

Building upon the theory developed for this equation, an extensive body of literature has formed around other infinity-Laplacian type problems which arise as $p$-Laplacian limits. 
For instance, one should mention the eigenvalue problem \cite{juutinen1999eigenvalue,esposito2015neumann,rossi2016first,bungert2022eigenvalue}, vector-valued problems \cite{sheffield2012vector,katzourakis2012variational}, problems with gradient constraints and limits of $p$-Poisson equations \cite{juutinen2016discontinuous,bhattacharya1989limits,bungert2023inhomogeneous}, and infinity-harmonic potentials \cite{lindgren2021gradient,brustad2023infinity}, where the list of references is far from being complete.

The present paper focuses on the standard infinity-Laplace equation \labelcref{eq:inf-Laplace_equation} and shall answer a question which appears to be entirely unexplored in the literature: 
\begin{center}
    How fast do the solutions of the $p$-Laplace equation \labelcref{eq:p-Laplace_equation} converge to the solution of the infinity-Laplace equation \labelcref{eq:inf-Laplace_equation} as $p\to\infty$?
\end{center}
The answer to this question was already revealed in the \rev abstract \nc and reads: 
At least as fast as $p^{-\nicefrac{1}{4}}$ tends to zero. 
In certain scenarios the rate \rev can \nc be improved to $p^{-\nicefrac{1}{2}}$ but it can never be better than $p^{-1}$, which is the convergence rate of the fundamental solutions to the $p$-Laplace equation to the one of the infinity-Laplace equation.

Most related to the present work are our papers \cite{bungert2022uniform,bungert2022ratio} in collaboration with Jeff Calder and Tim Roith.
There we used the ``comparison with cones'' property of infinity-harmonic functions together with techniques from \cite{smart2010infinity,armstrong2012finite} to prove rates of convergence for infinity-harmonic functions on sparse geometric graphs as the graph approximates a continuum.
Furthermore, in \cite{smart2010infinity,armstrong2012finite} the convergence rate for solutions of the inhomogeneous infinity Laplace equation $-\Delta_\infty u = \gamma$ as $\gamma\searrow 0$ was proven to be at most $\gamma^{\nicefrac{1}{3}}$.

Our key insight in \cite{bungert2022uniform,bungert2022ratio} was that graph infinity-harmonic functions satisfy a comparison principle with respect to functions that are close to a Euclidean cone, which enabled us to prove rates of convergence to an infinity-harmonic function.
In more detail, in \cite{bungert2022uniform} we utilize explicit error estimates between distance functions on a geometric graph and Euclidean cone functions to prove the rates. 
In \cite{bungert2022ratio} we use percolation theory to show that ratios of graph distances on sparse graphs converge to the corresponding ratio of Euclidean distances which is also sufficient for proving rates.
In the present paper, we transfer this line of thought to solutions of the $p$-Laplace equation for large $p$ and prove rates of convergence using explicit error estimates between ratios of Hölder cones $x\mapsto\abs{x}^{\frac{p-d}{p-1}}$ and the ratios of Euclidean (or Lipschitz) cones $x\mapsto\abs{x}$.

The rest of the paper is organized as follows:
In \cref{sec:setting_results} we introduce our notation, recap the concepts of solutions and comparison principles for equations \rev\labelcref{eq:p-Laplace_equation,eq:inf-Laplace_equation}\nc, state our main result, and discuss some extensions to the fractional infinity-Laplacian and equations on metric measure spaces.
\cref{sec:proof} is devoted to its proof: 
In \cref{sec:perturbations} we first recap important connections between infinity-harmonic functions and sub- and supersolutions of a finite difference infinity-Laplacian, as well as some perturbation results. 
In \cref{sec:approximate_consisteny} we prove an approximate sub- and supersolution property for $p$-harmonic functions with large values of~$p$ which is key for proving the rates in \cref{sec:convergence_rates}.

\section{Setting and main result}
\label{sec:setting_results}

\subsection{Notation}

In this paper we assume that $\Omega\subset\R^d$ is an open domain, we do not pose any regularity assumptions on its boundary.
With $\L^p(\Omega)$ and $\W^{1,p}(\Omega)$ for $p\in[1,\infty]$ we denote standard Lebesgue and Sobolev spaces, equipped with the norms $\norm{\cdot}_{\L^p}$ and $\norm{\cdot}_{\W^{1,p}}$, respectively.
Furthermore, for $p\in[1,\infty)$ the space $\W^{1,p}_0(\Omega)$ is defined as the closure of the space of compactly supported smooth functions with respect to the norm $\norm{\cdot}_{\W^{1,p}}$.

The space of continuous functions on $\closure\Omega$ is denoted by $\C(\closure\Omega)$ and equipped with the supremal norm $\norm{u}_\infty := \sup_{\closure\Omega}\abs{u}$. 
With $\C^{0,\alpha}(\closure\Omega)$ for $\alpha\in(0,1]$ we denote the space of $\alpha$-Hölder continuous functions, equipped with the norm $\norm{u}_{{0,\alpha}}:=\norm{u}_\infty + \seminorm{u}_{{0,\alpha}}$, where the Hölder semi-norm is defined as
\begin{align*}
    \seminorm{u}_{{0,\alpha}} := \sup\left\lbrace\frac{\abs{u(x)-u(y)}}{\abs{x-y}^\alpha} \st x,y\in\closure\Omega,\;x\neq y \right\rbrace.
\end{align*}
For $d<p<\infty$ the Sobolev spaces $\W^{1,p}_0(\Omega)$ are continuously embedded in the Hölder spaces $\C^{0,1-\nicefrac{d}{p}}(\closure\Omega)$, see \cite[Theorem 4.12]{adams2003sobolev}.

\subsection{Solutions of the \texorpdfstring{$p$}{p}-Laplace and infinity-Laplace equations}

We continue with defining solutions of the $p$-Laplace equation \labelcref{eq:p-Laplace_equation} and the infinity-Laplace equation \labelcref{eq:inf-Laplace_equation}.
\begin{definition}\label{def:solution_p-Laplace}
Let $p>d$ and $g_p\in\W^{1,p}(\Omega)\cap\C(\closure\Omega)$. 
We say that $u_p\in\W^{1,p}(\Omega)\cap\C(\closure\Omega)$ solves \labelcref{eq:p-Laplace_equation} if $u_p=g_p$ on $\partial\Omega$ and
\begin{align*}
    \int_\Omega\abs{\grad u_p}^p \dx \leq \int_\Omega\abs{\grad v}^p \dx\qquad\forall
    v\in\W^{1,p}(\Omega)\st
    v-g_p\in\W^{1,p}_0(\Omega).
\end{align*}
\end{definition}
\begin{definition}\label{def:solution_inf-Laplace}
Let $g_\infty\in\W^{1,\infty}(\Omega)\cap\C(\closure\Omega)$. 
We say that $u_\infty\in\C(\closure\Omega)$ solves \labelcref{eq:inf-Laplace_equation} if it is a viscosity solution of \labelcref{eq:inf-Laplace_equation}.
\end{definition}
\begin{remark}[Existence and uniqueness]
     Existence and uniqueness of solutions in the above sense are classical results.
     For solutions of the $p$-Laplace equation in the sense of \cref{def:solution_p-Laplace} we refer to \cite[Theorem 2.16, Section 3.1]{lindqvist2019notes}.
     As explained therein, the boundary of the domain can be arbitrarily irregular since we assume that the boundary data are continuous on $\closure\Omega$ and that $p>d$.     
     For an existence proof of solutions to the infinity-Laplace equation, constructed as limits of $p$-Laplacian solutions we refer to \cite[Theorem 4.6]{lindqvist2016notes}.
     Uniqueness was proved with quite different methods in \cite{barles2001existence,jensen1993uniqueness,armstrong2010easy}.
\end{remark}

Our analysis is entirely based upon the comparison principle of $p$-harmonic and infinity-harmonic functions with the respective fundamental solutions.

\begin{proposition}\label{prop:comparison}
    Let $u_p$ solve \labelcref{eq:p-Laplace_equation} for $d<p<\infty$ or \labelcref{eq:inf-Laplace_equation} for $p=\infty$.
    Furthermore, define the function
    \begin{align*}
        d_p(x,y) :=
        \begin{dcases}
            \abs{x-y}^\frac{p-d}{p-1},\quad&\text{if }p<\infty,\\
            \abs{x-y},\quad&\text{if }p=\infty.
        \end{dcases}
    \end{align*}
    Then for all domains $D\Subset\Omega$, compactly contained in $\Omega$, for all $a\geq 0$, and for all $x_0\in\R^d\setminus D$ it holds
    \begin{align*}
        \min_{\xi\in \partial D}
        \left\lbrace
        u(\xi) - a \,d_p(\xi,x_0)
        \right\rbrace
        \leq 
        u(x) - a\,d_p(x,x_0)
        \leq 
        \max_{\xi\in \partial D}
        \left\lbrace
        u(\xi) - a\,d_p(\xi,x_0)
        \right\rbrace
        ,
        \quad
        \forall x \in D.
    \end{align*}
\end{proposition}
\begin{proof}
    See \cite{lindqvist1986definition} for $p<\infty$ and \cite[Proposition 6.2]{lindqvist2016notes} or \cite{aronsson2004tour} for $p=\infty$.
\end{proof}
\begin{remark}
    The astonishing property of infinity-harmonic functions is that they are characterized through the comparison principle from \cref{prop:comparison}, see \cite{aronsson2004tour}. 
    This is not the case for $p$-harmonic functions, however, the comparison principle alone turns out to be enough.
\end{remark}

\subsection{Main result}
\label{sec:main_result}

The following theorem is our main result and provides a convergence rate which depends on the Hölder-regularity of the $p$-harmonic functions $u_p$ (which, a-priori, are at least $1-\nicefrac{d}{p}$-Hölder continuous).
\begin{theorem}[Explicit convergence rate]\label{thm:explicit_rate}
Let $u_p \in \W^{1,p}(\Omega)$ solve \labelcref{eq:p-Laplace_equation} for $p>d$ and $u_\infty \in \W^{1,\infty}(\Omega)$ solve \labelcref{eq:inf-Laplace_equation}.
Assume that $u_p \in \C^{0,\alpha_p}(\closure\Omega)$ for some $\alpha_p \in \left[1-\nicefrac{d}{p},1\right]$, and that 
\begin{align*}
    \mathsf H := \limsup_{p\to\infty}\seminorm{u_p}_{{0,\alpha_p}}
    <\infty.
\end{align*}
Then there exists a constant $C(\Omega,\mathsf H, \norm{u_\infty}_{{0,1}})\in(0,\infty)$ such that for all $p>d$ sufficiently large it holds that
\begin{align*}
    \norm{u_p - u_\infty}_{\infty}
    \leq 
    C(\Omega,\mathsf H, \norm{u_\infty}_{{0,1}})
    \left(
    \frac{d-1}{p-1}
    \right)^\frac{\alpha_p}{2\alpha_p+2}
    +
    \max_{\partial\Omega}\abs{g_p - g_\infty}.
\end{align*}
If $\essinf_\Omega\abs{\grad u_\infty}=:\gamma>0$, then this can be improved to
\begin{align*}
    \norm{u_p - u_\infty}_\infty
    \leq  
    \frac{C(\Omega,\mathsf H,\norm{u_\infty}_{{0,1}})}{\gamma^2}
    \left(\frac{d-1}{p-1}\right)^\frac{\alpha_p}{2}
    +
    \max_{\partial\Omega}\abs{g_p - g_\infty}.
\end{align*}
\end{theorem}

A couple of remarks on \cref{thm:explicit_rate} are in order.
\begin{remark}[Boundary term]\label{rem:boundary_term}
Since we are considering uniform convergence rates, measures through the supremal norm on $\closure\Omega$, the rate has to be dominated by the convergence rate of the boundary data, which explains the term $\max_{\partial\Omega}\abs{g_p-g_\infty}$.
Hence, the only interesting case is that the boundary data coincide, or converge quicker than the first term, such that the boundary term has no bearing. 
\end{remark}
\begin{remark}[Asymptotic rate]\label{rem:asymptotic_rate}
Let us now assume $g_p=g_\infty$ for simplicity.
Since $\alpha_p\to 1$ as $p\to\infty$, for large values of $p$ the convergence rates all behave like the rate for Lipschitz regularity with $\alpha_p=1$.
It holds
\begin{align*}
    \limsup_{p\to\infty}\norm{u_p-u_\infty}_\infty p^\frac{1}{4}
    < \infty
\end{align*}
and if $\essinf_\Omega\abs{\grad u_\infty}>0$ then 
\begin{align*}
    \limsup_{p\to\infty}\norm{u_p-u_\infty}_\infty p^\frac{1}{2}
    < \infty.
\end{align*}
This justifies the claim from the abstract of this paper which states that the convergence rate scales like $p^{-\nicefrac{1}{4}}$, respectively $p^{-\nicefrac{1}{2}}$ in the second case.
\end{remark}
\begin{remark}[The assumption $\mathsf H < \infty$]\label{rem:limsup_H}
Without any prior knowledge on the Hölder regularity of $u_p$, besides the trivial $\C^{0,1-\nicefrac{d}{p}}$ regularity, one can apply Morrey's inequality for $\alpha_p=1-\nicefrac{d}{p}$ to $v_p := u_p - g_p \in \W^{1,p}_0(\Omega)$ (see \cite[Lemma 2.3]{lindqvist2016notes}) to get
\begin{align*}
    \seminorm{v_p}_{{0,\alpha_p}}
    \leq 
    \frac{2pd}{p-d}
    \norm{\grad v_p}_{\L^p}.
\end{align*}
Taking the $\limsup$ as $p\to\infty$ we obtain, using also \cref{def:solution_p-Laplace}, that
\begin{align*}
    \limsup_{p\to\infty}\seminorm{v_p}_{{0,\alpha_p}}
    &\leq 
    2\limsup_{p\to\infty}
    \norm{\grad v_p}_{\L^p}
    \leq 
    2\limsup_{p\to\infty}
    \left(
    \norm{\grad u_p}_{\L^p}
    +
    \norm{\grad g_p}_{\L^p}
    \right)
    \\
    &\leq 
    4
    \limsup_{p\to\infty}
    \norm{\grad g_p}_{\L^p}.
\end{align*}
This implies that
\begin{align*}
    \mathsf H = 
    \limsup_{p\to\infty}\seminorm{u_p}_{{0,\alpha_p}}
    \leq 
    \limsup_{p\to\infty}
    \left(
    \seminorm{v_p}_{{0,\alpha_p}}
    +
    \seminorm{g_p}_{{0,\alpha_p}}
    \right)
    \leq 
    \limsup_{p\to\infty}
    \left(
    4\norm{\grad g_p}_{\L^p}
    +
    \seminorm{g_p}_{{0,\alpha_p}}
    \right),
\end{align*}
meaning that a uniform bound on semi-norms of the boundary data imply a uniform bound on the Hölder norms of $u_p$:
\begin{align*}
    \limsup_{p\to\infty}\norm{\grad g_p}_{\L^p} + \seminorm{g_p}_{{0,\alpha_p}}<\infty
    \implies 
    \mathsf H < \infty.
\end{align*}
\end{remark}
\begin{remark}[Parameter dependent rate]\label{rem:general_rate}
\cref{thm:explicit_rate} is a straightforward consequence of the more general statement \cref{thm:general_cvgc} further down in \cref{sec:convergence_rates} which asserts a convergence rate, depending on a free parameter $\eps>0$. 
Optimizing over this parameter leads to the explicit expressions in \cref{thm:explicit_rate}.
\end{remark}

\begin{example}[Lower bound]
There is no reason to assume that our rates are sharp. 
Still, the following example shows that the rate cannot be better than $\nicefrac{1}{p}$, \rev not even locally\nc.
To see this one considers the functions $u_p(x) = \abs{x}^\frac{p-d}{p-1}$ which solve \labelcref{eq:p-Laplace_equation} on the punctured ball $\Omega := \{x\in\R^d\st 0<\abs{x}<1\}$ with boundary values $g_p=1$ on the unit sphere and $g_p=0$ on the center. 
The solution of \labelcref{eq:inf-Laplace_equation} with the same boundary data $g_\infty=g_p$ is given by $u_\infty(x)=\abs{x}$.
Since this is a radial problem we have
\begin{align*}
    \norm{u_p - u_\infty}_\infty 
    =
    \max_{t\in[0,1]} t^\frac{p-d}{p-1} - t.
\end{align*}
Defining $\beta:=\tfrac{p-d}{p-1}\in(0,1)$ and $\phi(t) := t^\beta-t$ we see that $\phi(0)=\phi(1)=0$ and $\phi(t)>0$ for $t\in(0,1)$. 
Hence, the maximum is attained in the interior.
We observe that $\phi'(t) = 0$ is equivalent to $t = \beta^\frac{1}{1-\beta}$.
Plugging this into $\phi$ we get that
\begin{align*}
    \norm{u_p - u_\infty}_\infty 
    =
    \beta^\frac{\beta}{1-\beta} - \beta^\frac{1}{1-\beta}
    =
    \beta^\frac{1}{1-\beta}\left(\frac{1}{\beta}-1\right).
\end{align*}    
Using $\beta = \tfrac{p-d}{p-1} = 1-\tfrac{d-1}{p-1}$ we obtain
\begin{align*}
    \norm{u_p - u_\infty}_\infty 
    &=
    \left(\frac{p-d}{p-1}\right)^\frac{p-1}{d-1}
    \left(
    \frac{p-1}{p-d}-1
    \right)
    =
    \left(\frac{p-d}{p-1}\right)^\frac{p-1}{d-1}
    \frac{d-1}{p-d}
    \simeq \frac{1}{p},
\end{align*}    
meaning that in this example the rate is $\nicefrac{1}{p}$. 
This is better than $p^{-\nicefrac{1}{2}}$ which is guaranteed by \cref{thm:explicit_rate}.

\rev 
The lower bound of $\nicefrac{1}{p}$ cannot be improved, even if one considers local convergence rates. 
To see this, one can simply observe that
\begin{align*}
\abs{u_p\left(\frac{1}{2},0,\dots,0\right) - u_\infty\left(\frac{1}{2},0,\dots,0\right)} = \frac{1}{2^\beta} - \frac{1}{2} \geq \frac{\ln 2}{2}(1-\beta) = \frac{\ln 2}{2}\frac{d-1}{p-1} \sim \frac{1}{p}.
\end{align*}
Hence, for any smooth domain $\Omega'$ which is compactly contained in $\Omega$ and contains the point $(\nicefrac{1}{2},0,\dots,0)$ one will have $\norm{u_p-u_\infty}_{L^\infty(\Omega')}\sim \nicefrac{1}{p}$.
\nc
Judging from this example one might conjecture that the rate (at least for functions with positive gradient) is $\nicefrac{1}{p}$ in general.
\end{example}
\rev 
\begin{remark}
To obtain lower bounds for functions with vanishing gradient, which are expected to be worse than $\nicefrac{1}{p}$ a promising route might be to work with the $p$- and infinity-harmonic functions which were implicitly constructed in \cite{aronsson1984certain,aronsson1986construction}.
\end{remark}
\nc

\subsection{Extensions of our main result}
\label{sec:extension}

Let us discuss some extensions of our results.

We first note that the proof of \cref{thm:explicit_rate} in \cref{sec:proof} relies on only two properties of $p$-harmonic functions: comparison with Hölder cones and Hölder regularity. 
Therefore, it is straightforward to extend our result to other classes of PDE solutions which exhibit these properties, and we give two examples in the sequel.
Instead of $p$-harmonic functions, one could be interested in \rev harmonic function associated to \nc the $s$-fractional infinity-Laplacian, introduced in \cite{bjorland2012nonlocal} (see also \cite{del2022asymptotic,del2022evolution}).
For bounded and smooth functions $u:\R^d\to\R$ and $s\in(\tfrac12,1)$ it is defined as
\begin{align*}
    (\Delta_\infty)^s u(x) := 
    \int_0^\infty\frac{u(x+\eta\,v)-u(x-\eta\,v)-2u(x)}{\eta^{1+2s}}\d\eta,\qquad\text{where }v:=\frac{\grad u(x)}{\abs{\grad u(x)}}.
\end{align*}
For $\grad u(x)=0$ the definition has to be modified.
Solutions $u_s$ of the corresponding equations 
\begin{align}
    \begin{dcases}
        (\Delta_\infty)^s u = 0\quad&\text{in }\Omega,\\
        u = g_\infty\quad&\text{in }\R^d\setminus\Omega,
    \end{dcases}
\end{align}
satisfy a comparison principle with the Hölder cones $\abs{x}^{2s-1}$ and are $2s-1$-Hölder continuous if the boundary data $g_\infty$ are sufficiently regular \cite[Section 3]{bjorland2012nonlocal}.
Therefore, in the proof of \cref{thm:explicit_rate} in \cref{sec:convergence_rates} we can just plug $\alpha=\beta := 2s-1 \in (0,1)$ in \labelcref{eq:eps_in_terms_of_p} in order to obtain the rate
\begin{align*}
    \norm{u_s-u_\infty}_\infty
    \leq
    C(\Omega,\mathsf H,\norm{u_\infty}_{{0,1}})
    \left(1-s\right)^\frac{2s-1}{4s},
\end{align*}
where, as before, $u_\infty$ denotes the solution of \labelcref{eq:inf-Laplace_equation}.
Similarly, as for the $p$-Laplacian approximation discussed in \cref{sec:main_result} the rate asymptotically scales like $\left(1-s\right)^{\nicefrac14}$.
Similar arguments can be performed for the so-called Hölder infinity Laplacian equation
from \cite{chambolle2012holder} the solutions of which also comparison with Hölder cones as well as Hölder regularity.

Another extension of our results, albeit a less obvious one, concerns the convergence rate of $p$-harmonic functions on a metric measure space $(\Omega,d,\mu)$.
Such functions are defined as minimizers of a $p$-Dirichlet energy involving upper gradients, and admit a comparison principle with respect to generalized Green functions.
In the literature, the behavior of such Green functions has mainly been investigated through capacitary estimates in the singular case (which corresponds to $p\leq d$ in the Euclidean setting), see, for instance, \cite{holopainen2002singular,bjorn2020existence,bjorn2021volume}.
Although the non-singular case for large values of $p$ should be easier to treat, to extend our results one would need to prove the existence of a $p$-superharmonic Green function $u_p(x,x_0)$ with $u_p(x,x_0) = 0$ which admits the following ratio convergence
\begin{align*}
    \frac{\sup_{x\in B(x_0,r)}u_p(x,x_0)}{\inf_{x\not\in B(x_0,2r)}u_p(x,x_0)}
    -\frac{1}{2}
    \to 0\qquad\text{as }p\to\infty,\,r\to 0,
\end{align*}
including quantitative estimates for the convergence above in terms of $p$, the Ahlfors dimension of the measure, etc.

\section{Proof of the main result}
\label{sec:proof}

For $x\in\R^d$ and $\eps>0$ we let $B(x;\eps):=\Set{y\in\R^d\st\abs{x-y}\leq\eps}$ denote the \emph{closed ball} of radius $\eps$ around $x$.
Furthermore, we let $\Omega_\eps := \Set{x\in\Omega\st \inf_{y\in\partial\Omega}\abs{x-y}>\eps}$ denote the inner parallel set of $\Omega$ with distance $\eps$ to the boundary.

In the following we let $u:\closure\Omega\to\R$ be a function.
We define upper and lower envelopes of $u$ as
\begin{align}\label{eq:envelopes}
        u^\eps(x) := \sup_{y\in B(x;\eps)}u(y),
        \qquad
        u_\eps(x) := \inf_{y\in B(x;\eps)}u(y),
        \quad x\in \Omega_\eps.
\end{align}
We also define the upper and lower $\eps$-slopes of $u$ as
\begin{align}\label{eq:lower_slope}
    S_+^\eps u(x) := \frac{u^\eps(x)-u(x)}{\eps},
    \qquad
    S_-^\eps u(x) := \frac{u(x)-u_\eps(x)}{\eps},\qquad x\in\Omega_\eps.
\end{align}
A central tool for our techniques is the nonlocal (or finite difference) infinity-Laplacian which is defined as
\begin{align}\label{eq:nonlocal_inf-Laplace}
    \Delta_\infty^\eps u(x) 
    := 
    \frac{S_+^\eps u(x)
    -
    S_-^\eps u(x)}{\eps}
\end{align}
and can be expressed as
\begin{align*}
    \Delta_\infty^\eps u(x)
    =
    \frac{1}{\eps^2}
    \left[
    \sup_{y\in B(x;\eps)}
    \left(u(y)-u(x)\right)
    +
    \inf_{y\in B(x;\eps)}
    \left(u(y)-u(x)\right)
    \right],\quad x\in\Omega_\eps.
\end{align*}
From this expression it can be seen that for smooth functions $u$ formally $\Delta_\infty^\eps u$ is consistent with \rev the normalized infinity Laplacian $\tfrac{\langle \grad u,\hess u \, \grad u\rangle}{\abs{\grad u}^2}$ \nc as $\eps\to 0$.

\rev
We start by giving a brief outline of the proof strategy which relies on convenient properties of the nonlocal infinity-Laplacian $\Delta_\infty^\eps$ as well as the comparison principle for $p$-harmonic functions:
\begin{enumerate}
	\item Use established perturbation arguments to find $w$ which satisfies
	\begin{align*}
		-\Delta_\infty^\eps w \geq \delta^3,\qquad\abs{w-(u_\infty)_\eps}\lesssim\delta\qquad\text{in }\Omega_\eps.
	\end{align*}
	\item Use comparison with Hölder cones to prove that 
	\begin{align*}
		-\Delta_\infty^\eps u_p^\eps \leq \delta(\eps,p)^3
	\end{align*}
	for some $\delta=\delta(\eps,p)>0$ which also depends on the dimension $d$ and the Hölder constant and exponent of $u_p$.
	\item Use step 1., step 2., an elementary comparison principle for $\Delta_\infty^\eps$, and flip the signs to conclude that
	\begin{align*}
		\sup_{\closure{\Omega}_\eps}\abs{u_p^\eps-(u_\infty)_\eps} \lesssim \delta(\eps,p) + \sup_{\closure{\Omega}_\eps\setminus\closure{\Omega}_{2\eps}}\abs{u_p^\eps-(u_\infty)_\eps}.
	\end{align*}
	Use $\alpha$-Hölder continuity to get
	\begin{align*}
		\norm{u_p-u_\infty}_\infty \lesssim \eps^\alpha + \delta(\eps,p) + \norm{g_p-g_\infty}_\infty.
	\end{align*}
	\item Optimize over $\eps$ to obtain the final rate in terms of $p$.
\end{enumerate}

\begin{remark}[Alternative proof technique]
We are thankful for a reviewer pointing out that there might exist an alternative approach to proving convergence rates, based on the theory of viscosity solutions. 
Such an approach would rely on perturbation statements for the usual infinity-Laplacian $\Delta_\infty$ (see \cref{rem:perturb_1,rem:perturb_2} below) as well as the doubling-of-variables technique.
Since we believe that our approach generalizes easier to scenarios as in \cref{sec:extension} and is, furthermore, more elementary, we refrain from going into more detail here.
\end{remark}

\nc

\subsection{Perturbations of infinity-harmonic functions}
\label{sec:perturbations}

In this section we recall some important results which connect infinity-harmonic functions with sub- and supersolutions of the operator $\Delta_\infty^\eps$.
All results in this section were, to the best of our knowledge, first proved in the PhD thesis \cite{smart2010infinity} and picked up later, inter alia, in \cite{armstrong2010easy,armstrong2012finite,bungert2022uniform}.

We first recap the astonishing property of infinity-harmonic functions that their upper and lower envelopes are sub- or supersolutions associated to the operator $\Delta_\infty^\eps$ defined in \labelcref{eq:nonlocal_inf-Laplace}.
The proof is almost trivial and solely relies on the comparison with cones property of infinity-harmonic functions.

\begin{proposition}\label{lem:max-ball}
Let $u \in \C(\closure\Omega)$ solve \labelcref{eq:inf-Laplace_equation}
Then it holds for all $\eps>0$ that
\begin{align}
    -\Delta_\infty^\eps u^\eps \leq 
    0 \leq 
    -\Delta_\infty^\eps u_\eps\qquad \text{in }\Omega_{2\eps}.
\end{align}
\end{proposition}
\begin{proof}
The statement can be found in \cite[Theorem 2.2.3]{smart2010infinity} or \cite[Lemma 5]{armstrong2010easy}.
\end{proof}
The convenient property of the operator $\Delta_\infty^\eps$ is that it admits a comparison principle of the following form:
\begin{proposition}\label{prop:nonlocal_max_princ}
Assume that for a constant $C\geq 0$ the functions $u,v:\closure\Omega_\eps\to\R$ satisfy
\begin{align}
-\Delta_\infty^\eps u \leq C \leq -\Delta_\infty^\eps v
\end{align}
in $\Omega_{2\eps}$.
Then it holds
\begin{align}
    \sup_{\closure\Omega_\eps}(u-v) = \sup_{\closure\Omega_\eps\setminus\Omega_{2\eps}}(u-v).
\end{align}
\end{proposition}
\begin{proof}
The statement follows from \cite[Theorem 2.6.5]{smart2010infinity} or, more conveniently, from \cite[Proposition 3.3]{armstrong2012finite}. 
\end{proof}
In our final proof of convergence rates the role of $u$ will be played by the upper envelope of a $p$-harmonic function $u_p$ which we will prove to satisfy $-\Delta_\infty^\eps u_p^\eps \leq C$, where $C>0$ is a positive constant that depends on $\eps$, $p$, $d$, and the Hölder constants of $u$. 
So we need a comparison function $v$ which satisfies the inequality $-\Delta_\infty^\eps v \geq C$. 
Since according to \cref{lem:max-ball} the lower envelope of an infinity-harmonic function $u_\infty$ only satisfies $-\Delta_\infty^\eps (u_\infty)_\eps \geq 0$ we need to perturb it to a strict supersolution of this equation.

This can be achieved using the following lemmas which state that one can perturb any superharmonic function, associated to the operator $-\Delta_\infty^\eps$, into a superharmonic function $v$ with $S_-^\eps v$ bounded from below.
Further, one can perturb any such $v$ into a function $w$ which is a strict supersolution.
\begin{lemma}\label{lem:perturb_pos_grad}
If $u:\closure\Omega_\eps\to\R$ satisfies $-\Delta_\infty^\eps u \geq 0$ in $\Omega_{2\eps}$, then for any $\delta>0$ there is a function $v:\closure\Omega_\eps\to\R$ which satisfies
\begin{align*}
    -\Delta_\infty^\eps v\geq 0,
    \qquad 
    S_-^\eps v \geq {\delta},
    \qquad
    \text{and}
    \qquad
    u \leq v \leq u + 2 \delta \dist(\cdot,\Omega_\eps\setminus\Omega_{2\eps})\quad\text{in }\Omega_{2\eps}.
\end{align*}
\end{lemma}
\begin{proof}
    The proof can be found in a more general setting in \cite[Lemma 2.6.3]{smart2010infinity}.
\end{proof}
\rev 
\begin{remark}[\cref{lem:perturb_pos_grad} for $\eps=0$]\label{rem:perturb_1}
As pointed out by a reviewer, there exists a corresponding result for $\eps=0$ (cf. \cite[(6.6)]{armstrong2012finite} or \cite[Proposition 5.1]{crandall2008visit}) the proof of which we sketch in the following:
Considering $u$ which satisfies $-\Delta_\infty u\geq 0$ one defines $D_\delta:=\left\lbrace x\in\Omega\st \lim_{\eps\to 0}S_\eps^+u(x)<\delta\right\rbrace$ which turns out to be an open set.
Then one solves the eikonal equation
\begin{align*}
\abs{\grad \hat u} = \delta \quad\text{in }D_\delta
\qquad
\text{with}
\qquad
\hat u = u \quad\text{on }\partial D_\delta
\end{align*}
in the viscosity sense and sets
\begin{align*}
	v := 
	\begin{cases}
		\hat u \quad \text{in }D_\delta\\
		u \quad \text{in }\Omega\setminus D_\delta.
	\end{cases}
\end{align*}
The function can be shown to satisfy $-\Delta_\infty v\geq 0$ and $\abs{\grad v}\geq\delta$ in the viscosity sense and also $\abs{u-v}\leq C\delta$ for some domain-dependent constant $C>0$.

We note that the proof of \cref{lem:perturb_pos_grad} given in \cite[Lemma 2.6.3]{smart2010infinity} is a discrete version of the above argument.
Furthermore, as shown in the proof of \cite[Theorem 2.19]{armstrong2012finite} one can combine the result for $\eps=0$ with \cref{lem:max-ball} to obtain yet another proof of \cref{lem:perturb_pos_grad}.
\end{remark}
\nc

\begin{lemma}\label{lem:perturb_strict}
Suppose $v:\closure\Omega_\eps\to\R$ satisfies $-\Delta_\infty^\eps v \geq 0$ on $\Omega_{2\eps}$.
Then for all $0\leq\delta\leq\frac{1}{4\norm{v}_\infty}$ there exists a function $w:\closure\Omega_\eps\to\R$ that satisfies
\begin{align*}
    -\Delta_\infty^\eps w \geq -\Delta_\infty^\eps v + \delta(S^\eps_- v)^2\quad\text{on }\Omega_{2\eps}
    \qquad
    \text{and}
    \qquad
    \norm{v-w}_\infty \leq 3\norm{v}_\infty^2\delta.
\end{align*}
\end{lemma}
\begin{proof}
Very similar statements of this flavor can be found in \cite[Lemma 2.6.4]{smart2010infinity} or in \cite[Lemma 6.5]{armstrong2012finite}. 
However, to have all constants explicit we give the proof.

Without loss of generality we can assume $v\neq 0$.
Let us furthermore first assume that $u\geq 0$ in $\Omega$.
For $\delta\in\left[0,\frac{1}{2\norm{v}_\infty}\right]$ we define the function
\begin{align*}
    w := v - \delta v^2.
\end{align*}
Fix $x\in\Omega_{2\eps}$. 
Without loss of generality we allow ourselves to choose $x^\pm\in B(x;\eps)$ such that $v^\eps(x) = v(x^+)$ and $v_\eps(x) = v(x^-)$.
Otherwise, we can work with sequences of points which attain the supremum and infimum in the definition of $v^\eps$ and $v_\eps$. 
Note that the function $t\mapsto \lambda(t) := t - \delta t^2$ is monotone on $\left(-\infty,\frac{1}{2\delta}\right]$ and hence on the range of $v$.
This implies that $w^\eps(x) = w(x^+)$ and $w_\eps(x) = w(x^-)$.
By definition \labelcref{eq:nonlocal_inf-Laplace} it holds
\begin{align*}
    -\rev\eps\nc\Delta_\infty^\eps w(x) =  S_-^\eps w(x) - S_+^\eps w(x)
\end{align*}
and we will estimate each of these terms separately.
First, it holds
\begin{align*}
    \eps S_-^\eps w(x) 
    &=
    w(x) - w(x^-)
    \\
    &=
    v(x) - v(x^-)
    -
    \delta
    \left(
    v(x)^2-v(x^-)^2
    \right)
    \\
    &=
    \eps S_-^\eps v(x)
    -\delta
    \left(v(x)-v(x_-)\right)\left(v(x)+v(x_-)\right)
    \\
    &=
    \eps 
    S_-^\eps v(x)
    -
    \delta
    \eps
    S_-^\eps v(x)
    \left(2v(x) - \eps S_-^\eps v(x)\right)
\end{align*}
and therefore
\begin{align*}
    S_-^\eps w(x) = S_-^\eps v(x) - \delta S_-^\eps v(x)\left(2v(x) - \eps S_-^\eps v(x)\right).
\end{align*}
The other term is treated similarly:
Using that $v(x^+)\geq v(x)$ we obtain
\begin{align*}
    \eps S_+^\eps w(x)
    &=
    w(x^+) - w(x)
    \\
    &=
    v(x^+) - v(x) + \delta\left(v(x)^2-v(x^+)^2\right)
    \\
    &=
    \eps
    S_+^\eps v(x)
    +
    \delta
    \left(v(x)-v(x^+)\right)
    \left(v(x)+v(x^+)\right)
    \\
    &=
    \eps
    S_+^\eps v(x)
    -
    \delta
    \eps
    S_+^\eps v(x)
    \left(v(x)+v(x^+)\right)
    \\
    &\leq 
    \eps
    S_+^\eps v(x)
    -
    2
    \delta
    \eps
    S_+^\eps v(x)
    \,
    v(x)
\end{align*}
and hence
\begin{align*}
    S_+^\eps w(x) 
    \leq 
    S_+^\eps v(x) - 2\delta S_+^\eps v(x)\,v(x).
\end{align*}
Putting things together we obtain
\begin{align*}
    -\eps\Delta_\infty^\eps w(x)
    &\geq 
    S_-^\eps v(x) - \delta S_-^\eps v(x)\left(2v(x) - \eps S_-^\eps v(x)\right)
    -
    \left(
    S_+^\eps v(x) - 2\delta S_+^\eps v(x)\,v(x)
    \right)
    \\
    &=
    S_-^\eps v(x) 
    - 
    S_+^\eps v(x)
    +
    \delta\eps 
    \left(S_-^\eps v(x)\right)^2
    -
    2\delta
    S_-^\eps v(x)\,v(x)
    +
    2\delta S_+^\eps v(x)\,v(x)
    \\
    &=
    -\eps\Delta_\infty^\eps v(x)
    +\delta\eps\left(S_-^\eps v(x)\right)^2
    +2\delta
    \left(
    S_+^\eps v(x)
    -
    S_-^\eps v(x)
    \right)
    v(x)
    \\
    &=
    -\eps\Delta_\infty^\eps v(x)
    +\delta\eps\left(S_-^\eps v(x)\right)^2
    -2\delta
    \eps 
    \Delta_\infty^\eps v(x)\,
    v(x)
    \\
    &\geq 
    -\eps\Delta_\infty^\eps v(x)
    +\delta\eps\left(S_-^\eps v(x)\right)^2,
\end{align*}
using that $-\Delta_\infty^\eps v(x)\geq 0$ and $v\geq 0$.
Dividing by $\eps$ proves the claim.

For getting rid of the assumption $v\geq 0$ we let $v:\closure\Omega_\eps\to\R$ be arbitrary and consider $\Tilde{v}:=v+L\geq 0$, where $L:=\norm{v}_\infty$.
Applying the previous result shows that $\Tilde{w}:=\Tilde{v}-\delta\Tilde{v}^2$ satisfies for $\delta\in\left[0,\frac{1}{2\norm{\tilde v}_\infty}\right]$
that
\begin{align}\label{ineq:tilde_lapl}
    -\Delta_\infty^\eps\Tilde{w}
    \geq
    -\Delta_\infty^\eps\Tilde{v}
    +
    \delta
    (S_-^\eps\Tilde{v})^2.
\end{align}
Note that we can expand
\begin{align*}
    \Tilde{w} 
    = 
    v+L - \delta
    (v+L)^2
    =
    \underbrace{
    (1-2\delta L)v - \delta v^2}_{=:w}
    +
    L
    -
    \delta
    L^2.
\end{align*}
Hence, \labelcref{ineq:tilde_lapl} is equivalent to
\begin{align*}
    -\Delta_\infty^\eps{w}
    \geq
    -\Delta_\infty^\eps{v}
    +
    \delta
    (S_-^\eps{v})^2.
\end{align*}
Furthermore, we obtain
\begin{align*}
    \norm{v-w}_\infty 
    &=
    \norm{
    2\delta L v + \delta v^2
    }_\infty
    \leq 
    3\norm{v}_\infty^2\delta.
\end{align*}
Finally, since $\norm{\Tilde{v}_\infty}\leq 2\norm{v}_\infty$ the restriction $\delta\in\left[0,\frac{1}{2\norm{\Tilde{v}}_\infty}\right]$ is implied by $\delta\in\left[0,\frac{1}{4\norm{v}_\infty}\right]$.
This concludes the proof.
\end{proof}
\rev
\begin{remark}[\cref{lem:perturb_strict} for $\eps=0$]\label{rem:perturb_2}
Similar to the previous remark it should be mentioned that one can prove a corresponding result for $\eps=0$ by taking a solution of $-\Delta_\infty v\geq 0$ and perturbing it to a strict supersolution by defining $w:=v-\delta v^2$.
\end{remark}
\nc

\subsection{Approximate consistency of \texorpdfstring{$p$}{p}-harmonic functions}
\label{sec:approximate_consisteny}

In this section we shall prove an approximate version of \cref{lem:max-ball} for $p$-harmonic functions. 
Since they do not admit comparison with cones, but rather comparison with Hölder cones of the form $\abs{x}^\frac{p-d}{p-1}$, their upper and lower envelopes are only approximate sub- and supersolutions.

\begin{proposition}\label{prop:approx_consistency}
Let $u_p \in \W^{1,p}(\Omega)$ solve \labelcref{eq:p-Laplace_equation} for $p>d$, and assume that $u_p \in \C^{0,\alpha}(\Omega)$ for some $\alpha \in \left[1-\frac{d}{p},1\right]$.
Then it holds for all $\eps\in\left(0,\frac{1}{2}\right)$ that
\begin{subequations}
    \begin{align}
        -\Delta_\infty^\eps u_p^\eps(x) &\leq 
        \phantom{-}
        2^{1+\alpha} \seminorm{u_p}_{{0,\alpha}}\eps^{\alpha-2}\left(\frac{1}{2^\frac{p-d}{p-1}}-\frac12\right)
        ,\qquad x \in \Omega_{2\eps},\\
        -\Delta_\infty^\eps (u_p)_\eps(x) &\geq 
        -
        2^{1+\alpha} \seminorm{u_p}_{{0,\alpha}}\eps^{\alpha-2}\left(\frac{1}{2^\frac{p-d}{p-1}}-\frac12\right)
        ,\qquad x \in \Omega_{2\eps}.
    \end{align}
\end{subequations}
\end{proposition}
\begin{proof}
For a lighter notation we omit the subscript and write $u$ instead of $u_p$ and $g$ instead of $g_p$.
It suffices to prove the first statement. The second one is obtained by replacing $u$ with $-u$ (and $g$ in \labelcref{eq:p-Laplace_equation} with $-g$).

Let us fix $x\in\Omega_{2\eps}$. 
By \labelcref{eq:nonlocal_inf-Laplace,eq:envelopes} we can estimate
\begin{align}
    -\eps^2\Delta_\infty^\eps u^\eps(x) 
    &= 
    2u^\eps(x) - \max_{y\in B(x;\eps)}\max_{z\in B(y;\eps)}u(z)
    -
    \min_{y\in B(x;\eps)}\max_{z\in B(y;\eps)}u(z)
    \notag
    \\
    &\leq 
    2u^\eps(x) - \max_{y\in B(x;2\eps)}u(y)
    -
    u(x)
    \notag
    \\
    \label{eq:NL-inf-L-estimate}
    &=
    2u^\eps(x) - u^{2\eps}(x)
    -
    u(x).
\end{align}
We define $d_p(x,y) := \abs{x-y}^\frac{p-d}{p-1}$ and shall use the abbreviation $\beta:=\tfrac{p-d}{p-1}\in(0,1)$.
We will use the comparison principle to prove that
\begin{align}\label{eq:comparison_with_Holder_cones}
    u(y) \leq u(x) +\frac{u^{2\eps}(x)-u(x)}{\inf_{\abs{z-x}\geq 2\eps}d_p(x,z)}d_p(x,y)\qquad\forall y \in B(x;2\eps).
\end{align}
First, note that for $y\in\partial B(x;2\eps)$ it holds $d_p(x,y)\geq\inf_{\abs{z-x}\geq2\eps}d_p(x,z)$ and $u^{2\eps}(x)\geq u(y)$.
Second, it trivially it also holds $u(x)\leq u(x)$.
These two statements prove that
\begin{align*}
   u(y) \leq u(x) +\frac{u^{2\eps}(x)-u(x)}{\inf_{\abs{z-x}\geq 2\eps}d_p(x,z)}d_p(x,y)\qquad\forall y \in \partial B(x;2\eps)\cup\{x\}.
\end{align*}
Using the comparison principle from \cref{prop:comparison} on the open domain $D:=B(x;2\eps)\setminus\left(\{x\}\cup\partial B(x;2\eps)\right)$ with $\closure D = B(x;2\eps)\subset\Omega$ we see that \labelcref{eq:comparison_with_Holder_cones} holds true.

Next, we maximize \labelcref{eq:comparison_with_Holder_cones} over all $y\in B(x;\eps)$ to obtain $u^\eps(x)$ on the left hand side:
\begin{align*}
    u^\eps(x) 
    \leq 
    u(x)
    +
    \left(u^{2\eps}(x)-u(x)\right)
    \frac{\max_{y\in B(x;\eps)}d_p(x,y)}{\min_{\abs{z-x}\geq2\eps}d_p(x,z)}.
\end{align*}
The ratio on the right side of this expression can be explicitly computed as
\begin{align*}
    \frac{\max_{y\in B(x;\eps)}d_p(x,y)}{\min_{\abs{z-x}\geq2\eps}d_p(x,z)}
    =
    \frac{\eps^\beta}{(2\eps)^\beta}
    =\frac{1}{2^\beta}.
\end{align*}
Hence, we obtain
\begin{align*}
    u^\eps(x) 
    &\leq 
    u(x)
    +
    \left(u^{2\eps}(x)-u(x)\right)
    \frac{1}{2^\beta}
    \\
    &=
    u(x)
    +
    \left(u^{2\eps}(x)-u(x)\right)
    \left(\frac{1}{2} + \frac{1}{2^\beta} - \frac{1}{2}\right)
    \\
    &=
    \frac{1}{2}
    \left(
    u(x) + u^{2\eps}(x)
    \right)
    +
    \seminorm{u}_{{0,\alpha}}
    (2\eps)^\alpha
    \left(\frac{1}{2^\beta} - \frac{1}{2}\right).
\end{align*}
Plugging this into \labelcref{eq:NL-inf-L-estimate} and dividing by $\eps^2$ we obtain
\begin{align*}
    -\Delta_\infty^\eps u^\eps(x)
    \leq 
    2
    \seminorm{u}_{{0,\alpha}}
    (2\eps)^\alpha
    \eps^{-2}
    \left(\frac{1}{2^\beta} - \frac{1}{2}\right)
    =
    2^{1+\alpha}
    \seminorm{u}_{{0,\alpha}}
    \eps^{\alpha-2}
    \left(\frac{1}{2^\beta} - \frac{1}{2}\right).
\end{align*}
\end{proof}

\subsection{Convergence rates}
\label{sec:convergence_rates}

Now we have proved all we need for proving convergence rates. 
We first prove the following convergence rate which depends on a free parameter $\eps>0$.
Optimizing over this parameter will then yield \cref{thm:explicit_rate}.

\begin{theorem}[General convergence rate]\label{thm:general_cvgc}
Let $u_p \in \W^{1,p}(\Omega)$ solve \labelcref{eq:p-Laplace_equation} for $p>d$ and $u_\infty \in \W^{1,\infty}(\Omega)$ solve \labelcref{eq:inf-Laplace_equation}.
Assume that $u_p \in \C^{0,\alpha}(\closure\Omega)$ for some $\alpha \in \left[1-\nicefrac{d}{p},1\right]$.
Let $0 < \eps < \nicefrac{1}{2}$ and $p>d$ so large such that 
\begin{align}\label{eq:restriction_p_eps}
    \frac{1}{2^\frac{p-d}{p-1}} - \frac12 
    \leq 
    \frac{\eps^{2-\alpha}}{2^{7+\alpha}\norm{u_\infty}_\infty^3\seminorm{u_p}_{{0,\alpha}}}.
\end{align}
Then there exists a constant $C=C(\Omega,\norm{u_\infty}_\infty)$ such that 
\begin{align}\label{eq:general_rate}
    \begin{split}
    \norm{u_p - u_\infty}_{\infty}
    &\leq 
    (2+2^{\alpha})\seminorm{u_p}_{{0,\alpha}}\eps^{\alpha}
    +
    4[u_\infty]_{{0,1}}\eps
    \\
    &\qquad
    +
    C
    \left( \seminorm{u_p}_{{0,\alpha}}\eps^{\alpha-2}\left(\frac{1}{2^\frac{p-d}{p-1}}-\frac12\right)\right)^\frac{1}{3}
    +
    \max_{\partial\Omega}\abs{g_p - g_\infty}.
    \end{split}
\end{align}
If $\essinf_\Omega\abs{\grad u_\infty}=:\gamma>0$, then this can be improved to
\begin{align}\label{eq:general_rate_pos_grad}
    \begin{split}
    \norm{u_p - u_\infty}_{\infty}
    &\leq 
    (2+2^{\alpha})\seminorm{u_p}_{{0,\alpha}}\eps^{\alpha}
    +
    4[u_\infty]_{{0,1}}\eps
    \\
    &\qquad
    +
    C
    \seminorm{u_p}_{{0,\alpha}}\eps^{\alpha-2}\left(\frac{1}{2^\frac{p-d}{p-1}}-\frac12\right)
    +
    \max_{\partial\Omega}\abs{g_p - g_\infty}.
    \end{split}
\end{align}
\end{theorem}
\begin{remark}
The restriction \labelcref{eq:restriction_p_eps} on $p$ can be satisfied if $\alpha=\alpha_p$ and $\limsup_{p\to\infty}\seminorm{u_p}_{{0,\alpha_p}}<\infty$, see \cref{rem:limsup_H} for a sufficient condition for this to hold.
\end{remark}
\begin{proof}[Proof of \cref{thm:general_cvgc}]
We will only prove an upper bound for $\max_{\closure\Omega}(u_p - u_\infty)$. 
The converse inequality follows by replacing $u_p$ and $u_\infty$ with $-u_p$ and $-u_\infty$, respectively.
By \cref{lem:max-ball,prop:approx_consistency} it holds 
\begin{align*}
    -\Delta_\infty^\eps u_p^\eps \leq 
    2^{1+\alpha} \seminorm{u_p}_{{0,\alpha}}\eps^{\alpha-2}\left(\frac{1}{2^\frac{p-d}{p-1}}-\frac12\right)
    \qquad 
    \text{and}
    \qquad 
    -\Delta_\infty^\eps (u_\infty)_\eps \geq 0
    \qquad
    \text{in }\Omega_{2\eps}.
\end{align*}
By assumption \labelcref{eq:restriction_p_eps} it holds
\begin{align*}
    \delta:=\left(2^{1+\alpha} \seminorm{u_p}_{{0,\alpha}}\eps^{\alpha-2}\left(\frac{1}{2^\frac{p-d}{p-1}}-\frac12\right)\right)^\frac{1}{3} \leq \frac{1}{4\norm{u_\infty}_\infty}
\end{align*}
and hence we can apply \cref{lem:perturb_pos_grad,lem:perturb_strict} with this value of $\delta$ to find a function $w:\closure\Omega_{2\eps}\to\R$ such that 
\begin{align*}
    -\Delta_\infty^\eps u_p^\eps \leq
    \delta^3
    &\leq 
    -\Delta_\infty^\eps w
    \qquad
    \text{in }\Omega_{2\eps}
    \qquad
    \text{and}
    \qquad 
    \norm{w - (u_\infty)_\eps}_\infty
    \leq 
    \Tilde C(\Omega,\norm{u_\infty}_\infty)
    \delta.
\end{align*}
Note that the constant has the explicit value $\Tilde C(\Omega,\norm{u_\infty}_\infty)=2\diam(\Omega)+3\norm{u_\infty}_\infty^2$, albeit we do not claim this is optimal.
Together with \cref{prop:nonlocal_max_princ} this implies that
\begin{align*}
    \max_{\closure\Omega_{\eps}}(u_p^\eps - (u_\infty)_\eps) 
    &\leq 
    \sup_{\closure\Omega_{\eps}}(u_p^\eps - w) 
    +
    \sup_{\closure\Omega_{\eps}}(w - (u_\infty)_\eps) 
    \\
    &\leq     \sup_{\closure\Omega_{\eps}\setminus\Omega_{2\eps}}(u_p^\eps - w)
    +
    \Tilde C(\Omega,\norm{u_\infty}_\infty)\delta
    \\
    &\leq     \max_{\closure\Omega_{\eps}\setminus\Omega_{2\eps}}(u_p^\eps - (u_\infty)_\eps)
    +
    2
    \Tilde C(\Omega,\norm{u_\infty}_\infty)\delta.
\end{align*}
In order to complete the proof, we have to replace the upper and lower envelopes on the left hand side by $u_p$ and $u_\infty$, and control the boundary term on the right hand side.
Both terms are treated easily, using that $u_p\in\C^{0,\alpha}(\closure\Omega)$ and $u_\infty\in\C^{0,1}(\closure\Omega)$.
Using this we obtain the estimate
\begin{align*}
    \max_{\closure\Omega}(u_p-u_\infty)
    &\leq 
    \seminorm{u_p}_{{0,\alpha}}\eps^\alpha
    +
    [u_\infty]_{{0,1}}\eps
    +
    \max_{\closure\Omega_\eps}(u_p-u_\infty), 
    \\
    &\leq 
    \seminorm{u_p}_{{0,\alpha}}\eps^\alpha
    +
    [u_\infty]_{{0,1}}\eps
    +
    \max_{\closure\Omega_\eps}(u_p^\eps-(u_\infty)_\eps).
\end{align*}
Similarly we can estimate
\begin{align*}
    \max_{\closure\Omega_{\eps}\setminus\Omega_{2\eps}}(u_p^\eps - (u_\infty)_\eps)
    &\leq 
    \seminorm{u_p}_{{0,\alpha}}\eps^\alpha
    +
    [u_\infty]_{{0,1}}\eps
    +
    \max_{\closure\Omega_{\eps}\setminus\Omega_{2\eps}}(u_p-u_\infty)
    \\
    &\leq
    (1+2^\alpha)\seminorm{u_p}_{{0,\alpha}}\eps^\alpha
    +
    3[u_\infty]_{{0,1}}\eps
    +
    \max_{\partial\Omega}\abs{g_p - g_\infty}.
\end{align*}
Combining all these estimates we obtain
\begin{align*}
    \max_{\closure\Omega}(u_p-u_\infty)
    \leq 
    (2+2^\alpha)\seminorm{u_p}_{{0,\alpha}}\eps^\alpha
    +
    4[u_\infty]_{{0,1}}\eps
    +
    2
    \Tilde C(\Omega,\norm{u_\infty}_\infty)\delta
    +
    \max_{\partial\Omega}\abs{g_p - g_\infty}.
\end{align*}
Using the definition of $\delta$ and defining $C(\Omega,\norm{u_\infty}_\infty):=2^\frac{4+\alpha}{3}\Tilde C(\Omega,\norm{u_\infty}_\infty)$ concludes the proof of the first statement.

In the case that $\essinf_\Omega\abs{D u_\infty}=:\gamma>0$ we can apply \cref{lem:perturb_strict} with
\begin{align*}
    \delta:=\frac{2^{1+\alpha}}{\gamma^2}\seminorm{u_p}_{{0,\alpha}}\eps^{\alpha-2}\left(\frac{1}{2^\frac{p-d}{p-1}}-\frac12\right)
\end{align*}
directly to $v := (u_\infty)_\eps$.
Let us for now claim that $S_-^\eps v \geq \gamma$ and, hence, we obtain a function $w\in\C(\closure\Omega_\eps)$ such that
\begin{align*}
    -\Delta_\infty^\eps u_p^\eps 
    \leq 
    \gamma^2\delta
    \leq 
    -\Delta_\infty^\eps w
    \qquad
    \text{in }\Omega_{2\eps}
    \qquad\text{and}
    \qquad 
    \norm{w - (u_\infty)_\eps}_\infty
    \leq 
    C(\Omega,\norm{u_\infty}_\infty)
    \delta.
\end{align*}
From here the proof continues as above, albeit with the different value of $\delta$ and with the constant $C(\Omega,\norm{u_\infty}_\infty) := 2^{2+\alpha}\Tilde{C}(\Omega,\norm{u_\infty}_\infty)$.

It remains to argue that indeed $S_-^\eps v \geq \gamma$.
Remembering that $v = (u_\infty)_\eps$ and that, in particular, $-\Delta_\infty u_\infty \geq 0$, \cite[Lemma 5.2]{armstrong2012finite} implies that $S_-^\eps u_\infty \geq \gamma$. 
This is not yet quite enough, but as in \cite[Proof of Theorem 2.19]{armstrong2012finite} one obtains $S_-^\eps(u_\infty)_\eps \geq S_-^\eps u_\infty\geq\gamma$.
This concludes the proof.
\end{proof}

By optimizing over the free parameter $\eps>0$ we obtain the explicit convergence rate in \cref{thm:explicit_rate}.
For this one balances the first and the third term in \labelcref{eq:general_rate} to express $\eps$ in terms of $p$. 
Note that the second term of order $\eps$ is dominated by the first $\eps^{\alpha}$ since $\alpha\leq 1$.
The final rate is the sum of the resulting value of $\eps^{\alpha}$ and the error of the boundary data.
\begin{proof}[Proof of \cref{thm:explicit_rate}]
We start by optimizing the right hand side in \labelcref{eq:general_rate} in terms of $\eps$. 
The sum of all $\eps$-dependent terms will be as small as possible if they all scale in the same way.
Since $\alpha\leq 1$ and $\eps<1$, the term $\eps^\alpha$ dominates the term $\eps$ and it suffices to choose $\eps$ such that the following is satisfied:
\begin{align*}
    \eps^\alpha = \left(\eps^{\alpha-2}\left(\frac{1}{2^\beta}-\frac12\right)\right)^\frac{1}{3}
\end{align*}
where we abbreviate $\beta:=\frac{p-d}{p-1}$, like in the proof of \cref{thm:general_cvgc}.
We can equivalently reformulate this equation as follows:
\begin{align*}
    &\phantom{\iff}
    \eps^\alpha 
    = 
    \left(\eps^{\alpha-2}\left(\frac{1}{2^\beta}-\frac12\right)\right)^\frac{1}{3}
    \\
    &\iff 
    \eps^{2\alpha+2} 
    = 
    \frac{1}{2^\beta}-\frac12
    \\
    &\iff 
    \eps
    = 
    \left(
    \frac{1}{2^\beta}-\frac12
    \right)^\frac{1}{2\alpha+2}.
\end{align*}
The function $(0,1)\ni\beta \mapsto \frac{1}{2^\beta}-\frac12$ is squeezed between two linear functions
\begin{align*}
    \frac{\ln 2}{2}(1-\beta) \leq \frac{1}{2^\beta}-\frac12 \leq \frac{1}{2}(1-\beta)
\end{align*}
and so we do not loose anything but bounding it from above with the right linear function, yielding
\begin{align}\label{eq:eps_in_terms_of_p}
    \eps \leq \left(\frac{1-\beta}{2}\right)^\frac{1}{2\alpha+2}
    =
    \left(\frac12\frac{d-1}{p-1}\right)^\frac{1}{2\alpha+2},
\end{align}
where we resubstituted $\beta$.
        
Since the right hand side of this expression converges to zero as $p\to\infty$ and $\mathsf H<\infty$, for $p$ large enough \labelcref{eq:restriction_p_eps} is satisfied.
Hence, we can apply \cref{thm:general_cvgc} and
\labelcref{eq:eps_in_terms_of_p} to obtain the rate
\begin{align*}
    \norm{u_p - u_\infty}_\infty
    &\leq 
    C(\Omega,\mathsf H,\norm{u_\infty}_{{0,1}})
    \eps^\alpha
    +
    \max_{\partial\Omega}\abs{g_p - g_\infty}
    \\
    &\leq 
    C(\Omega,\mathsf H,\norm{u_\infty}_{{0,1}})
    \left(
    \frac{d-1}{p-1}
    \right)^\frac{\alpha}{2\alpha+2}
    +
    \max_{\partial\Omega}\abs{g_p - g_\infty}.
\end{align*}
In the case that $\essinf_\Omega\abs{\grad u_\infty}=:\gamma>0$ we have to solve
\begin{align*}
    \eps^\alpha = \eps^{\alpha-2}\left(\frac{1}{2^\beta}-\frac12\right).
\end{align*}
Even easier than before one gets
\begin{align*}
    \eps \leq 
    \left(\frac12\frac{d-1}{p-1}\right)^\frac{1}{2},
\end{align*}
which results in the rate
\begin{align*}
    \norm{u_p - u_\infty}_\infty
    \leq 
    \frac{C(\Omega,\mathsf H,\norm{u_\infty}_{{0,1}})}{\gamma^2} 
    \left(\frac{d-1}{p-1}\right)^\frac{\alpha}{2}
    +
    \max_{\partial\Omega}\abs{g_p - g_\infty}.
\end{align*}
\end{proof}

\section*{Acknowledgments}

This work was supported by the Deutsche Forschungsgemeinschaft (DFG, German Research Foundation) under Germany's Excellence Strategy - GZ 2047/1, Projekt-ID 390685813. 
Parts of this work were also done while the author was in residence at Institut Mittag-Leffler in Djursholm, Sweden during the semester on \textit{Geometric Aspects of Nonlinear Partial Differential Equations} in 2022, supported by the Swedish Research Council under grant no. 2016-06596.
The author would also like to thank Mikko Parviainen and Tim Roith for interesting discussions on the topic of this paper as well as Peter Lindqvist and Antoni Kijowski for helpful remarks on the first version of the preprint.

\printbibliography

\end{document}